\documentclass[reqno, 12pt]{amsart}

\usepackage{a4wide, amsmath, amscd, amsthm, amssymb}
\usepackage[utf8]{inputenc}
\usepackage[french,colorlinks,citecolor=blue]{hyperref}
\usepackage{tikz}
\usepackage[english]{babel}
\newtheorem{theorem}{Theorem}[section]

\numberwithin{equation}{section}

\newcommand{\ZZ}{\mathbb{Z}}

\newcommand{\QQ}{\mathbb{Q}}

\def\K{\mathbb{K}}

\def\la{\lambda}

\def\s{\sigma}
\newtheorem{thm}{Theorem}[section]

\newtheorem{defn}[thm]{Definition}
\theoremstyle{remark}

\newcommand\des{\mathop{\rm des}}
\newcommand{\Des}{\operatorname{Des}}

\newcommand{\g}{{\gamma}}

\def\s{ \sigma }

\def\k+1th{(k+1)^{th}}

 \DeclareMathOperator{\fmaj}{fmaj}
\DeclareMathOperator{\col}{col}

\def\fmaj{\mathop{\rm fmaj}\nolimits}

 \def\maj{\mathop{\rm maj}\nolimits}

\def\des{\mathop{\rm des}\nolimits}

\def\il{\bigl]\kern-.55em\bigl]}
\def\ir{\bigr]\kern-.55em\bigr]}

\newcommand*\pFq[5]{{}_{#1}\Phi_{#2}
\biggl(\genfrac..{0pt}{}{#3}{#4};q,#5\biggr)}

\begin{document}

\title{Moments of $q$-Jacobi polynomials and $q$-zeta values}

\author{Fr\'ed\'eric Chapoton}
\address[Fr\'ed\'eric Chapoton]{Institut de Recherche Mathématique Avanc\'ee
Universit\'e de Strasbourg
7 rue Ren\'e Descartes
F-67084 STRASBOURG Cedex
FRANCE}
\email{chapoton@math.unistra.fr}

\author{Christian Krattenthaler$^\dagger$}
\address[Christian Krattenthaler]{Fakult\"at f\"ur Mathematik
Universit\"at Wien
Oskar-Morgenstern-Platz 1
A-1090 Vienna
AUSTRIA}
\email{christian.krattenthaler@univie.ac.at}

\author{Jiang Zeng}
\address[Jiang Zeng]{Institut Camille Jordan, CNRS UMR 5208,
 Université Claude Bernard Lyon~1,
 Université de Lyon, 
    F-69622 Villeurbanne cedex, France}                   
\email{zeng@math.univ-lyon1.fr}

\footnote{Research partially supported by the Austrian
Science Foundation FWF (grant S50-N15)
in the framework of the Special Research Program
``Algorithmic and Enumerative Combinatorics"}

\date{\today}
%\setlength{\topmargin}{0cm}
%\setlength{\headheight}{0cm}
%\setlength{\headsep}{0cm}
%\setlength{\textwidth}{16.5cm}
%\addtolength{\textheight}{2cm}
%\setlength{\oddsidemargin}{0cm}

\begin{abstract}
  We explore some connections between moments of rescaled little
  $q$-Jacobi polynomials, $q$-analogues of values at negative integers
  for some Dirichlet series, and the $q$-Eulerian polynomials of wreath
  products of symmetric groups.
\end{abstract}

\maketitle

%\tableofcontents
\section*{Introduction}

This article is about a connection between three kinds of objects, namely
\begin{itemize}
  \item[(A)] $q$-analogues of Dirichlet series and their values at negative integers,
  \item[(B)] basic hypergeometric polynomials and their sequences of moments,
  \item[(C)] weighted enumeration of elements in coloured symmetric groups.
\end{itemize}

Let us give more details on these three points in order.

The point (A) is about a $q$-analogue of the Dirichlet series
\begin{equation}
  L(s,c,r) = \sum_{\substack{m\geq 1\\m\equiv c\,\, (\text{mod }r)}} \frac{1}{m^s},
\end{equation}
where $c, r$ are fixed integers. This is the Riemann zeta function
when $(c, r)=(1,1)$. For general $c$ and $r$, the summands do not
form a multiplicative sequence, so there is no Euler product. One defines as in ~\cite{chapoton2010fractions} a $q$-analogue of this Dirichlet series as an operator
\begin{equation}
  L_q(s,c,r) = \sum_{\substack{m\geq 1\\m\equiv c\,\, (\text{mod }r)}} \frac{1}{[m]_q^s} F_m,
\end{equation}
where $[m]_q = (q^m-1)/(q-1)$ is the usual $q$-integer and $F_m$ is
the formal Frobenius operator, acting on formal power series in $z$
with no constant term and coefficients in $\QQ(q)$, defined by
\begin{equation*}
  F_{m}(f)(q, z)=f(q^{m}, z^m).
\end{equation*}
Whenever the Dirichlet series $L(s,c,r)$ factorizes as an Euler
product, then so does the operator $L_q(s,c,r)$ as a product of
commuting operators.

One then introduces some $q$-analogues of the values of $L(s,c,r)$ at
non-positive integers, namely $L_q(-n,c,r)(z)$ for $n \geq 0$. As images of the formal power series $f(z)=z$, these are formal power series in
the variable $z$ with coefficients in $\QQ(q)$. As we will see, these
are in fact rational functions in $q$ and $z$.

The point (B) is about the little $q$-Jacobi polynomials, a system of
orthogonal polynomials in one variable. This is one of the families in
the Askey--Wilson scheme of basic hypergeometric orthogonal
polynomials (cf.\ \cite{koekoek2010hypergeometric}). 
The little $q$-Jacobi polynomials, orthogonal with
respect to the variable $x$, depend on the variable $q$ and two
further parameters. For each choice of integers $(c,r)$, by an
appropriate choice of these parameters and some affine change in the
variable $x$, one obtains a system of orthogonal polynomials involving
the variables $q$ and $z$. Their sequence of moments, which are
evaluations of the associated linear functional at the monomials
$x^n$, are therefore rational functions in $q$ and $z$.

The point (C) is about the complex reflection groups $G(r,n)$ defined as the
wreath product of a symmetric group $S_n$ by a cyclic group
$\ZZ_r$. The elements of these groups can be seen as coloured
permutation matrices, where non-zero entries contain a root of unity
of order dividing $r$. By using two combinatorial statistics on these
elements, one can refine the number $r^n n!$ of elements of $G(r,n)$ into
a polynomial in two variables $q$ and $z$, with positive integer
coefficients. In this context, the parameter $c$ is absent.

The aim of this article is to show that (A), (B) and (C) all give
essentially the same rational function in $q$ and $z$. More precisely,
the rational functions from (A) and (B) are essentially the quotients
of the polynomial from (C) by simple denominators. The part~(C) is
involved only when the parameter $c$ equals $1$.

The relationship between (C) and (A) is merely a reformulation of the
results by Biagioli and the third author in~\cite{biagioli2011enumerating}. The
relationship between (A) and (B) is a $(q,z)$-analogue of well-known
results about Bernoulli numbers and Euler numbers, for which we refer
to \cite{dilcher_lin_2}.

\section{Preliminaries}

\subsection{Orthogonal polynomials}

In this subsection we recall some fundamental results of the theory of
orthogonal polynomials~\cite{chihara2011introduction, viennot_notes}. Let $\K$ be a field.

\begin{defn}
  Let $\varphi : \K[x]\to \K$ be a linear functional.  A sequence of
  polynomials $\{p_n(x)\}_{n\geq 0}$ in $\K[x]$ is said to be
  \emph{orthogonal} with respect to the linear functional $\varphi$ if:
  \begin{itemize}
  \item[(i)] $p_n(x)$ is of degree $n$, for $n=0,1,\dots$;
  \item[(ii)] $\varphi(p_n(x)\,p_{n'}(x))=K_n\,\delta_{n,n'}$, $K_n\not=0$,
for $n=0,1,\dots$.
  \end{itemize} The sequence $\{\mu_n\}_{n\geq 0}$ with
  $\mu_n=\varphi(x^n)$ for $n\geq 0$ is called the \emph{moment sequence
    associated} with $\varphi$.
\end{defn}

Sometimes the polynomials $\{p_n(x)\}$ are also  said to be orthogonal with respect to the sequence of moments $\{\mu_n\}_{n\geq 0}$.

Let us write OPS as a shorthand for \textit{orthogonal polynomial system}.

\begin{theorem}[\sc Favard's theorem]\label{favard}
  A sequence of polynomials $\{p_n(x)\}_{n\geq 0}$ in $\K[x]$ is a monic
  OPS if and only if there is  a sequence $\{b_n\}_{n\geq 0}$ and a non-zero
  sequence $\{\la_n\}_{n\geq 0}$ such that $p_0(x)=1$, $p_1(x)=x-b_0$ and
  \begin{equation}\label{three}
    p_{n+1}(x)=(x-b_n)p_n(x)-\la_{n}p_{n-1}(x)\quad \textrm{for}\quad n\geq 1.
  \end{equation}
\end{theorem}

\begin{theorem}\label{shift-three-term}
Let the polynomials $(p_n(x))_{n\ge0}$ satisfy \eqref{three}.
Then, if we fix $\mu_0:=\lambda_0\neq 0$, the functional
  $\varphi$ with respect to which this OPS is orthogonal is unique.
Furthermore, for
  $Q_n(x)=\alpha^{-n}p_n(\alpha x+\beta)$, $\alpha\neq 0$, we have
  \begin{align}
    Q_{n+1}(x)=\left(x-\frac{b_n-\beta}{\alpha}\right)Q_{n}(x)-\frac{\lambda_{n}}{\alpha^2}Q_{n-1}(x), \quad \textrm{for}\quad n\geq 1,
  \end{align}
  and, if $(p_n(x))_{n\ge0}$ is  the OPS with respect to the moments $(\mu_n)$, then
  $(Q_n(x))$ is the  OPS with respect to the moments $\nu_n$ given by
  \begin{eqnarray}\label{shift-mom}
    \nu_n=\varphi\left(\left(\frac{x-\beta}{\alpha}\right)^n\right)=\alpha^{-n}\sum_{j=0}^n\binom{n}{j}(-\beta)^{n-j}\mu_j, \quad \textrm{for}\quad n\geq 0.
  \end{eqnarray}
\end{theorem}

\begin{theorem}\label{cf-moments}
  The generating function of the moments $\{\varphi(x^n)\}$ has the continued fraction expansion
  \begin{align}\label{eq:cfrac}
    \sum_{n\geq 0}\varphi(x^n)\,t^n=
    \cfrac{\la_0}{1-b_0 t-\cfrac{\la_1 t^2}{1-b_1t-\cfrac{\la_2 t^2}{\ddots}}}.
  \end{align}
\end{theorem}

There is also an associated formula for Hankel determinants of the
sequence of moments, see \cite{viennot_notes, kratt_det1, kratt_det2}

\subsection{Wreath  product of a symmetric group by a cyclic group}

Let $r \geq 1$ and $n \geq 1$ be integers. Let $S_n$ be the symmetric group on $\{1,\dots,n\}$. A permutation $\s \in S_n$ will be denoted by $\s=\s(1)\cdots \s(n).$ The {\em wreath product} $\ZZ_r \wr S_n$ of
$\mathbb{Z}_r$ by $S_n$ is the set
\begin{equation}\label{def-grn}
G(r,n):=\{(c_1,\ldots,c_n;\sigma) \mid c_i \in \{0,\ldots,r-1\},\sigma \in
S_n\}.
\end{equation}
Using a fixed primitive $r$-th root of unity $\xi$, one can see the elements in this set as square matrices, starting from the permutation matrix for $\sigma$ and replacing the non-zero entry in column $i$ by $\xi^{c_i}$.

This group is therefore also called the {\em group of r-coloured
  permutations}. We will represent its elements as
\begin{equation*}
  \g = [\g(1),\ldots,\g(n)]=[\s(1)^{c_1},\ldots, \s(n)^{c_n}].
\end{equation*}
We denote by
\begin{equation*}
  \col(\g): =\sum_{i=1}^n c_i,
\end{equation*}
the {\em colour weight} of any $\g \in G(r,n)$. For example, if $\g=[4^1, 3^0, 2^4, 1^2] \in G(5,4)$ then $\col(\g)=7$.
\smallskip

We endow the set of possible values for the $\g(i)$ with the following total order:
\begin{equation*}
   n^{r-1} < \dots < n^1 < \dots < 1^{r-1} < \dots < 1^{1} < 0 < 1^0 < \dots < n^0.
\end{equation*}
The $0$ is inserted here to separate the ``positive" values $i^0$ from 
the ``negative" values $i^c$ with $c\ge1$. It will also be used in
the statistics that we are going to define now.

The {\em descent set} of $\g \in G(r,n)$ is defined by 
\begin{eqnarray}\label{des}
\Des_G(\g)
&:=&\{i \in \{0,\ldots,n-1\} \mid \g(i) > \g(i+1)\},
\end{eqnarray}
where  $\g(0):=0$, and its cardinality is denoted by $\des_G(\g)$.
\smallskip

The {\em major index}  is defined  to be the sum of descent positions:
$$ \maj(\g)=\sum_{i\in \Des_G(\g)}i,$$
and the \emph{flag-major index} is defined by $$\fmaj(\g):=r \cdot \maj(\g)+ \col(\g).$$
\smallskip
For example, for  $\g=[4^1, 3^0, 2^4, 1^2] \in G(5,4)$ we have $\Des_{G}(\g)=\{0,2\}$, $\des_{G}(\g)=2$, $\maj(\g)=2$, and $\fmaj(\g)=17$.

Biagioli and the third author~\cite{biagioli2011enumerating} defined 
the generating polynomials for $G(r,n)$ with respect to
the bi-statistic  $(\des, \fmaj)$:
\begin{equation}
G_{r,n}(Z,q)=\sum_{\g\in G(r,n)} Z^{\des_G(\g)}q^{\fmaj(\g)},
\end{equation}
and they proved the following identity.
We refer the reader to Subsection~\ref{sec:1.3} for the meaning of the
$q$-notations.

\begin{thm}[\sc Carlitz--MacMahon identity for $G(r,n)$]\label{e:ca}
Let $r$ and $ n $ be positive integers. Then
\begin{align}\label{eq:carlitzG}
\frac{G_{r,n}(Z,q)}{(Z;q^r)_{n+1}}=
\sum_{k\geq 0}Z^k [r  k +1]_q^n.
\end{align}
\end{thm}
The above formula gives a nice generalization of identities of 
Carlitz~\cite{carlitz1975combinatorial} for the symmetric group
(corresponding to the case where $r=1$), and of Chow and Gessel~\cite{chow2007descent} for the hyperoctahedral group (corresponding to the case where $r=2$).  

\subsection{Little $q$-Jacobi polynomials}
\label{sec:1.3}

We use the standard $q$-notations from \cite{koekoek2010hypergeometric}, among which
\begin{equation*}
[x]_q=\frac{1-q^x}{1-q},
\end{equation*}
the $q$-Pochhammer symbol
\begin{equation*}
  (a ; q)_n = (1 - a)(1- aq)\ldots(1-a q^{n-1}),
\end{equation*}
and the convenient shorthand
\begin{equation*}
  (a, b ; q)_n = (a ; q)_n\, (b ; q)_n.
\end{equation*}
We furthermore
need the $q$-binomial theorem~\cite[p.~16]{koekoek2010hypergeometric}
\begin{align}
  \label{q-binomial}
  {}_1\Phi_0(a; -; q, z)=\sum_{k=0}^\infty 
  \frac{(a;q)_k}{(q;q)_k}z^k=\frac{(az;q)_\infty}{(z;q)_\infty}
\end{align}
and the $q$-Chu--Vandermonde formula~\cite[p.~17]{koekoek2010hypergeometric}
\begin{align}
  \label{q-chu}
  {}_2\Phi_1(q^{-n}, b; c; q, q)
=\sum_{k=0}^\infty 
  \frac{(q^{-n};q)_k\,(b;q)_k}{(c;q)_k\,(q;q)_k}q^k
=\frac{(c/b;q)_n}{(c;q)_n}b^n.
\end{align}

The little $q$-Jacobi 
polynomials~\cite[p.~482]{koekoek2010hypergeometric} have the explicit representation
\begin{equation}\label{little jacobi}
  p_n(x;a,b\mid q) = \pFq{2}{1}{q^{-n},abq^{n+1}}{aq}{qx}
=
\sum_{k=0}^\infty 
  \frac{(q^{-n};q)_k\,(abq^{n+1};q)_k}{(aq;q)_k\,(q;q)_k}(qx)^k,
\end{equation}

and 
are 
orthogonal with respect to
the inner product  defined by
\begin{equation*}
\int_0^1f(x)g(x)d_qw(x)=\sum_{k=0}^\infty f(q^k)g(q^k) w(q^k),
\end{equation*}
where 
\begin{equation*}
w(x)= \frac{(aq,bq;q)_\infty}{(abq^2,q;q)_\infty}\cdot  \frac{(qx;q)_\infty}{(bqx;q)_\infty}
x^{\alpha+1}
\end{equation*}
with $a=q^\alpha$.

Let $p_n(x)$ be the monic little $q$-Jacobi polynomials, i.e., 
\begin{equation*}
  p_n(x;a,b\mid q) =\frac{(-1)^n q^{-\binom{n}{2}} (abq^{n+1}; q)_n}{(aq;q)_n} p_n(x).
\end{equation*}
Then the normalized  recurrence relation~\cite[p.~483]{koekoek2010hypergeometric} reads
\begin{align}\label{monic-three-term-monic-jacobi}
  xp_n(x)=p_{n+1}(x)+(A_n+C_n) p_n(x)+A_{n-1}C_n p_{n-1}(x),
\end{align}
where
\begin{align*}
  A_n&=q^n \frac{(1-aq^{n+1})(1-abq^{n+1})}{(1-abq^{2n+1})(1-abq^{2n+2})},\\
  C_n&=aq^n \frac{(1-q^n)(1-bq^{n})}{(1-abq^{2n})(1-abq^{2n+1})}.
\end{align*}
By the $q$-binomial theorem~\eqref{q-binomial}, the $n$th moment is
\begin{equation}\label{mom-jacobi}
  \mu_n=\int_0^1x^n d_qw(x)=\frac{(aq;q)_n}{(abq^2; q)_n},\quad 
\text{for }n=0, 1, 2, \ldots.
\end{equation}
We can also verify \eqref{mom-jacobi}  by using the explicit formula \eqref{little jacobi} and the $q$-Chu-Vandermonde formula~\eqref{q-chu}: 
namely, for $n\geq 1$, we have
\begin{equation}
\int_0^1  p_n(x;a,b\mid q) d_qw(x)=0.
\end{equation}

We can now prove the connection between (B) and (A).
\begin{theorem}
  \label{moments-shift-little-jacobi}
  For integers $r\geq 1$,
  the $n$th moment $\mu_n$ of the shifted little $q$-Jacobi polynomials  
  $p_n(q^{-c}(1+(q-1)x) ; Zq^{-r}, 1\mid q^r)$ is 
  \begin{align}\label{r-moments}
    \mu_n=(1-Z)\sum_{k\geq 0}([rk+c]_q)^nZ^k.
  \end{align}
  For $c=1$, we have
  \begin{equation}
    \mu_n = \frac{G_{r,n}(Z,q)}{(Z q^r;q^r)_n}.
  \end{equation}
\end{theorem}
\begin{proof}
  By \eqref{shift-mom},  the $n$th moment of $p_n(q^{-c}(1+(q-1)x); a, b\mid q^r)$ is
  \begin{equation*}
    \nu_n=q^{nc}(q-1)^{-n}\sum_{j=0}^n \binom{n}{j} (-q^c)^{j-n} \frac{(aq^r;q^r)_j}{(abq^{2r}; q^r)_j}.
  \end{equation*}
  Substituting $a$ by $Z q^{-r}$ and  $b$ by $1$, we get
  \begin{align*}
    \nu_n&=q^{nc} (q-1)^{-n}\sum_{j=0}^n \binom{n}{j} (-q^c)^{j-n} \frac{1-Z}{1-Z q^{rj}}\\
         &=(1-Z)(q-1)^{-n}\sum_{k\geq 0} Z^k\sum_{j=0}^n \binom{n}{j} (-1)^{n-j}   q^{(rk+c)j}\\
         &=(1-Z)\sum_{k\geq 0}([rk+c]_q)^nZ^k.
  \end{align*}
  The last statement follows from \eqref{eq:carlitzG}.
\end{proof}

%%%%%%%%%%%%%%%%

\begin{theorem}
The generating function for the moments $\mu_n$ in \eqref{r-moments} has the continued fraction expansion
\begin{align}\label{eq:cfrac2}
\sum_{n\geq 0}\mu_n t^n=
\cfrac{1}{1-b_0 t-\cfrac{\la_1 t^2}{1-b_1t-\cfrac{\la_2 t^2}{\ddots}}}.
\end{align}
where the coefficients $b_n$ and $\la_n$ are given by 
\begin{equation}
  \la_n= \frac{Z q^{2r(n-1)+2 c}\,[rn]_q^2\, (1-Z q^{r(n-1)})^2}{(1-Z q^{2rn})(1-Z q^{r(2n-1)})^2(1-Z q^{r(2n-2)})}
\end{equation}
and 
\begin{equation}
  b_n= \frac{q^c}{q-1}
  \left(\frac{q^{rn}(1-Z q^{rn})^2}{(1-Z q^{2rn})(1-Z q^{r(2n+1)})}+\frac{Z q^{r(n-1)}(1-q^{rn})^2}{(1-Z q^{r(2n-1)})(1-Z q^{2rn})}-q^{-c}\right).
\end{equation}
\end{theorem}
\begin{proof}
This follows by combining \eqref{monic-three-term-monic-jacobi} and  Theorems~\ref{moments-shift-little-jacobi}, \ref{shift-three-term} and \ref{cf-moments} with $\alpha=(q-1)/q^c$ and $\beta=q^{-c}$.
\end{proof}
%%%%%

% The last identity is Carlitz's identity~\eqref{eq:carlitzG}.\cite{carlitz1975combinatorial}
%\begin{equation}
%  \frac{1}{(z,q)_{n+1}}\sum_{\sigma \in S_n} q^{\maj(\sigma)} z^{\des(\sigma)}
%  = \sum_{k \geq 0} z^{k} [k+1]_q^n.
%\end{equation}
%Although this identity is
%usually attributed to Carlitz~\cite{carlitz1975combinatorial}, it is actually a special case of a result of MacMahon~\cite[Vol. 2, Chapter 4]{macmahon2001combinatory}.

%%%%%%%%%%%%%%%%%

\section{Zeta operators at negative integers}
\label{sec:2}

We define the $q$-difference operator on the formal power series $f$
in $z$ by
\begin{equation}
  \Delta_z(f) = \frac{f(q z) - f(z)}{q - 1}.
\end{equation}
Note that $\Delta_{z}(z^m)=[m]_q z^m$. This implies that 
repeated application of the
operator $\Delta_z$ creates the sequence of values at negative integers for
the $q$-analogues of Dirichlet series. Indeed, for $n \geq 0$, we have
\begin{equation}
  \label{value_L_minusn}
  L_q(-n,c,r)(z) =  \sum_{\substack{m\geq 1\\m\equiv\,\, c (\text{mod }r)}} [m]_q^{n} z^m.
\end{equation}
and therefore
\begin{equation}
  \label{Delta_on_L_values}
  \Delta_z \left(L_q(-n,c,r)(z)\right) = L_q(-n-1,c,r)(z).
\end{equation}

Computing the initial value for $n=0$, one finds
\begin{equation}
  \label{value_L_zero}
  L_q(0,c,r)(z) = \frac{z^c}{1-z^r}.
\end{equation}
By induction using \eqref{Delta_on_L_values}, the expression $L_q(-n,c,r)(z)$ is a rational function
in $q$ and $z$ with denominator $(z^r, q^r)_{n+1}$.

The general relation between (A) and (B) is therefore, by comparison
of \eqref{value_L_minusn} with \eqref{r-moments}, using \eqref{value_L_zero}, that
\begin{equation}
  \label{connexion_AB}
   L_q(-n,c,r)(z) / L_q(0,c,r)(z) = \mu_n\big| _{Z=z^r}.
\end{equation}

For $c=1$, comparison with \eqref{eq:carlitzG} reveals the
combinatorial expression
\begin{equation}
  L_q(-n,c,r)(z) = z \frac{G_{r,n}(z^r,q)}{(z^r,q^r)_{n+1}},  % ok
\end{equation}
which makes the precise connection between (A) and (C).

\bibliographystyle{plain}
%\nocite{*}
\bibliography{Chapo-Kratt-Zeng}
%\input{Chapo-Kratt-Zeng.bbl}
%%%%%%%%%%%%%%%%%%%%%%%%%%%%%%%%%%%%%%%%%%%%%%%%%%%%%%%%%%%%
\end{document}